\documentclass[reqno, 11pt,a4paper]{amsart}

\usepackage{fullpage}  
\usepackage{calc,graphicx,amsfonts,amsthm,amscd,epsfig,psfrag,amsmath,amssymb,enumerate,dsfont}
\usepackage{subfig} 
\usepackage{pdfsync}
\usepackage{stmaryrd} 
\usepackage[numeric,initials,nobysame]{amsrefs}

\usepackage[colorlinks=true, pdfstartview=FitV, linkcolor=blue, citecolor=blue, urlcolor=blue,pagebackref=false]{hyperref}

\usepackage[showlabels,sections,floats,textmath,displaymath]{}
\newcommand{\ie}{\hbox{\it i.e.\ }}

\reversemarginpar
\newlength\fullwidth
\numberwithin{equation}{section}

\DeclareMathSymbol{\leqslant}{\mathalpha}{AMSa}{"36} 
\DeclareMathSymbol{\geqslant}{\mathalpha}{AMSa}{"3E} 
\DeclareMathSymbol{\eset}{\mathalpha}{AMSb}{"3F}     
\renewcommand{\leq}{\;\leqslant\;}                   
\renewcommand{\geq}{\;\geqslant\;}                   

\newcommand{\1}{\mathds{1}}

\newcommand{\var}{\operatorname{Var}}

\newcommand{\D}{\Delta}

\renewcommand{\l}{\lambda}
\renewcommand{\L}{\Lambda}

\renewcommand{\l}{\lambda}
\renewcommand{\a}{\alpha}
\renewcommand{\d}{\delta}
\renewcommand{\t}{\tau}

\newcommand{\g}{\gamma}

\newcommand{\e}{\varepsilon}

\renewcommand{\O}{\Omega}

\newcommand{\tc}{\thinspace |\thinspace}

\newtheorem{theorem}{Theorem}[section]

\newtheorem{lemma}[theorem]{Lemma}

\newtheorem{remark}[theorem]{Remark}

\newtheorem{definition}[theorem]{Definition}

\newcommand{\N}{\mathbb N}

\newcommand{\cA}{\ensuremath{\mathcal A}}

\newcommand{\cC}{\ensuremath{\mathcal C}}
\newcommand{\cD}{\ensuremath{\mathcal D}}

\newcommand{\cG}{\ensuremath{\mathcal G}}
\newcommand{\cH}{\ensuremath{\mathcal H}}

\newcommand{\cK}{\ensuremath{\mathcal K}}
\newcommand{\cL}{\ensuremath{\mathcal L}}

\newcommand{\cP}{\ensuremath{\mathcal P}}

\newcommand{\cU}{\ensuremath{\mathcal U}}

\newcommand{\bbE}{{\ensuremath{\mathbb E}} }

\newcommand{\bbP}{{\ensuremath{\mathbb P}} }

\newcommand{\bbR}{{\ensuremath{\mathbb R}} }

\newcommand{\bbZ}{{\ensuremath{\mathbb Z}} }

\newcommand{\E}{\mathbb{E}}

\let\a=\alpha    \let\d=\delta  \let\e=\varepsilon
 \let\g=\gamma \let\h=\eta      \let\l=\lambda
          \let\p=\pi  
  \let\s=\sigma \let\t=\tau   
  
\let\D=\Delta     \let\L=\Lambda 
\let\O=\Omega      

\renewcommand{\le}{\leq}
\newcommand{\be}{\mathbf{e}}

\newcommand{\Tmix}[1]{{T_{\textrm{mix}}^{(#1)}}}

\newcommand{\Ep}[2]{\bbE_{#1}\left[#2\right]}  
\newcommand{\fc}{\mathbf{1}} 
\newcommand{\Var}{{\rm Var}}

\title{Mixing time bounds for oriented kinetically constrained spin models}
\author[P. Chleboun]{P. Chleboun$^1$}
 \address{$^1$Dip. Matematica,  Universit\`a  Roma Tre, Largo S.L.Murialdo 00146, Roma, Italy.}
\email{paul@chleboun.co.uk}

\author[F. Martinelli]{F. Martinelli$^1$}
\email{martin@mat.uniroma3.it}
 
\date{}
\thanks{\sl Work supported by the European Research Council through the ``Advanced
Grant'' PTRELSS 228032.}

\begin{document}

\begin{abstract}
We analyze the mixing time of a class of oriented kinetically
constrained spin models (KCMs) on a $d$-dimensional lattice of $n^d$
sites. A typical example is the North-East model, a $0$-$1$ spin
system on the two-dimensional integer lattice that evolves according
to the following rule: whenever a site's southerly and westerly
nearest neighbours have spin $0$, with rate one it resets its own spin
by tossing a $p$-coin, at all other times its spin remains
frozen. Such models are very popular in statistical physics because,
in spite of their simplicity, they
display some of the key features of the dynamics of real glasses. We
prove that the mixing time is
$O(n \log n)$ whenever the relaxation time is $O(1)$.  
Our study was motivated by the ``shape'' conjecture put forward by
G. Kordzakhia and S.P. Lalley.\\


\end{abstract}

\maketitle

\section{Introduction}
Kinetically constrained spin models (KCMs) are interacting
$0$-$1$ particle systems, on general graphs, which evolve with a simple Glauber dynamics
described as follows. At every site $x$ the system tries to update the occupancy variable (or spin) at
$x$ to the value $1$ or $0$ with probability $p$ and $q$
respectively. However the
update at $x$ is accepted only if the \emph{current} local
configuration satisfies a certain constraint, hence the models are ``kinetically
constrained''. It is always assumed that the constraint at site $x$ does
not to depend on the spin at $x$ and therefore the product
Bernoulli($p$) measure $\pi$ is the reversible measure.  Constraints
may require, for example, that a certain \emph{number} of the
neighbouring spins are in state $0$, or more restrictively, that certain
\emph{preassigned} neighbouring spins are in state $0$ (e.g. the
children of $x$ when the underlying graph is a rooted tree). 

The main interest in the physical literature for KCMs (see e.g. 
\cite{Ritort} for a review) stems from the
fact  that they display many key dynamical
features of real glassy materials: ergodicity breaking
transition at some critical value $q_c$, huge relaxation time for $q$
close to $q_c$, dynamic heterogeneity (non-trivial spatio-temporal fluctuations of the
local relaxation to equilibrium) and aging, just to mention a
few. Mathematically, despite their simple definition, KCMs pose
very challenging and interesting problems because of the hardness of
the constraint, with ramifications towards bootstrap percolation
problems \cite{Spiral}, combinatorics \cite{CDG,Valiant:2004cb},
coalescence processes \cite{FMRT-cmp,FMRT} and  random walks on
upper triangular matrices \cite{Peres-Sly}. Some of the mathematical tools
developed for the analysis of the relaxation process of KCMs
\cite{CMRT} proved to
be quite powerful also in other contexts such as card shuffling problems
\cite{Bhatnagar:2007tr} and random evolution of surfaces \cite{PietroCaputo:2012vl}.

In this paper we focus on \emph{oriented} KCMs on a $d$-dimensional
lattice, $d\geq 2$, with $n^d$ sites, in particular on their mixing
time. A prototypical model belonging to the above class
of KCMs is the North-East model in two dimensions (see
e.g. \cite{Kordzakhia:2006} and \cite{CMRT}) for which
the constraint at any given site $x$ requires the south \emph{and} west neighbours
of $x$ to be empty in order for a flip at $x$ to occur. In order to
avoid trivial irreducibility issues the south-westerly most
spin is unconstrained and sites outside the upper quadrant are treated
as fixed zeros.

With $p_c$ the 
percolation threshold for oriented percolation in two dimensions (see
e.g. \cite{Durrett:1984tm}), it was 
proved in \cite{CMRT} that for all $p<p_c$ the relaxation time  of the North-East process
is $O(1)$ while it becomes $\O(e^{cn})$\footnote{We recall that
  $f=\O(g)$ if $|f|\ge c g$ for some $c>0$. }, for some $c>0$, when $p>p_c$. At $p=p_c$ the relaxation time is expected to have a poly$(n)$
growth. Consider now the North-East model in the first
quadrant of $\bbZ^2$. In
\cite{Kordzakhia:2006} it was conjectured that, for $p<p_c$ and
starting from all $1$'s, the
\emph{influence region} $R_t$, defined as the union of all unit
squares around those sites which
have flipped at least once by time $t$, has a definite limiting shape $S\subset \bbR^2$ in the
sense that a.s.  $\frac{R_t}{t}\to S$ as $t\to \infty$. Since the
North-East process is neither monotone or additive (see
\cite{Liggett1}), the usual tools to prove a
shape theorem do not apply in this case. 

The above conjecture implies that, for $p<p_c$, the influence coming
from the unconstrained spin at the South-West corner
propagates at a definite linear rate as it does in the East model \cite{Blondel:2012wb}, the
one dimensional analog of the model (for background 
see \cite{Aldous,CMRT,SE1}). In particular the mixing time of
the model should grow linearly in $n$ (the linear size of the system). However in 
dimension $d\ge 2$ the analysis of the propagation of influence is quite delicate because of
the many paths along which it can occur (see \cite{Valiant:2004cb}
for combinatorial results in this direction).

In this paper we prove that the mixing time is $O(n\log n)$ as long as
the spectral gap of the process 
is $\O(1)$. Our technique bares some similarities to those
employed in \cite{Caputo:2011dv} to analyse the Glauber dynamics of
biased plane partitions.

\section{Models and Results}
\subsection{Setting and notation}
We consider a class of $0$-$1$ interacting particle systems on
finite subsets $\L$ of the integer lattice $\bbZ^d$, reversible with respect to the
product measure $\p := \prod_{x\in\L}\p_x$, where $\pi_x$ is the
Bernoulli$(p)$ measure. 

The $d$-dimensional cube of linear size $n$ (which contain $n^d$
points) will be
denoted by
\begin{align*}
  \L_n := \left( [1,n] \times \ldots \times [1,n] \right) \cap \bbZ^d
  \,.
\end{align*}
The standard basis vectors in $\bbZ^d$ are denoted
$\be_1=(1,\ldots,0)$, $\be_1=(0,1,\ldots,0),\ldots$,
$\be_d=(0,\ldots,1)$. For $x \in \L_n$ we write $x_j$ for the component of $x$ in the
direction $\be_j$.

The set of probability measures on the finite state space $\O_n=\{0,1\}^{\L_n}$ is
denoted by $\cP(\O_n)$. Elements of $\O_n$ will be denoted by the
small greek letters $\s,\h,\ldots$ and $\s_x$ will denote the spin at
the vertex $x$. 

We denote the $i$-th level hyperplane in $\L_n$ by $\cH_i = \{ x \in
\L_n : \sum_{j=1}^{d}x_j = i\}$ (see Fig. \ref{fig:Geom}). 
The set of sites on and below this hyperplane will be written  $\cU_i = \{x \in \L_n : \sum_{j=1}^{d}x_j \leq i \}$.
For each $\sigma \in \O_n$ we
write $\sigma^{(i)}$ for the restriction of $\sigma$ to $\cU_i$, and
$\s_{\cH_i}$ for the restriction of $\s$ to $\cH_i$.
Similarly, for any probability measure $\nu\in \cP(\O_n)$, we write $\nu^{(i)}$ for the marginal of
$\nu$ on $\O^{(i)}:=\{0,1\}^{\cU_i}\subset \O_n$.

For any vertex $x\in \L_n$ we also let (see Fig. \ref{fig:Geom} Right)
\begin{align*}
  \cK_x^* &= \Big\{ y \in \bbZ^d : y = x - \sum_{i=1}^{d}\a_i\be_i, \
  \a_i \in \{1,0\} \Big\}\setminus \{ x\}\\
 \cK_x &= \Big\{ y \in \bbZ^d : \exists\  i=1,\dots ,d \text{ such that }y = x - \be_i\Big\}\,.
\end{align*}
\begin{figure}[t]
  \centering
  \mbox{\subfloat{\vtop{%
  \vskip0pt
  \hbox{\includegraphics[width=0.48\textwidth]{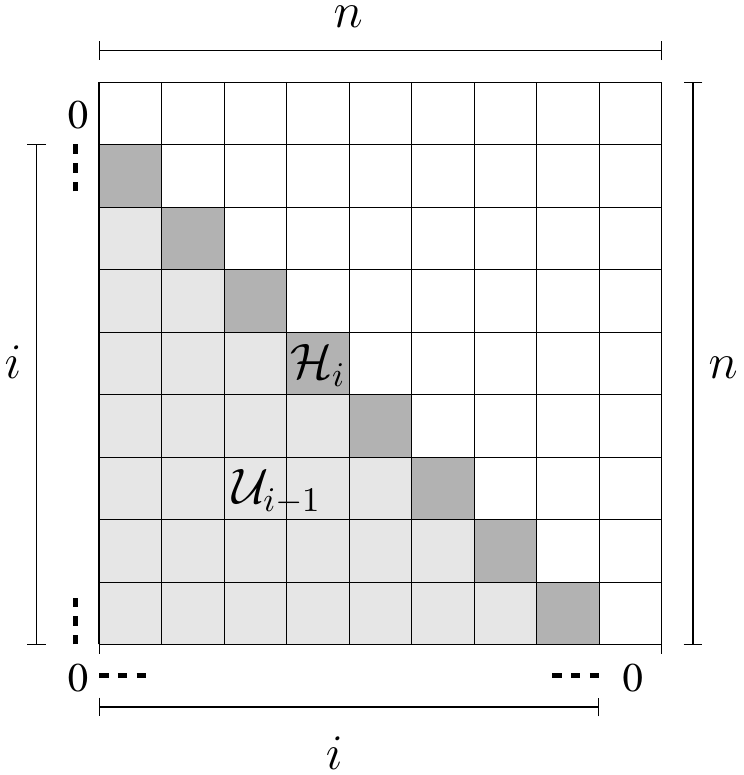}}}}\quad \quad 
    \subfloat{\vtop{%
  \vskip20pt
  \hbox{\includegraphics[width=0.18\textwidth]{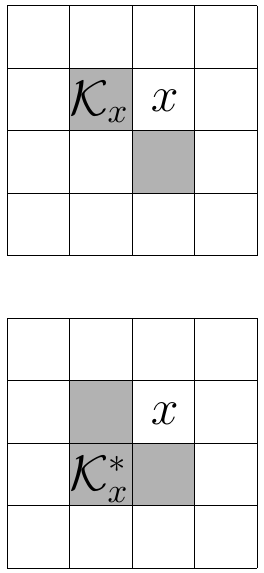}}}}}
  \caption{\label{fig:Geom} Left: The  the lattice $\Lambda_n$ in two
    dimensions, $d=2$. The sets $\cH_i$ and  $\cU_{i-1}$
    are shown in dark grey and light grey respectively. Right: The constraining neighbourhoods,
    $\cK_x$ and $\cK_x^*$, of  a vertex $x$ in two dimensions, shown
    in dark grey.}
\end{figure}
\begin{definition}[constraints]
Let $x_*=(1,\dots,1)$ be the south-west corner of $\L_n$. Consider a
collection $\{\cC_x\}_{x\in \L_n}$ of 
\emph{constraining
  neighborhoods} such that 
\begin{align*}
\cC_{x_*}=\emptyset,\qquad\text{and}\qquad 
\emptyset \neq
\cC_x\subset K_x^*\cap \L_n, \quad x\neq x_*.  
\end{align*}
Let 
\begin{align}
  \label{eq:constraint}
  c_x(\s) : =
  \begin{cases}
\prod_{y\in\cC_x}(1-\s_y)
&\text{if $x\neq x_*$}\\
1 &\text{$x=x_*$.}
  \end{cases}
\end{align}
We then say that
the constraint at site $x$ is satisfied by the configuration
$\s$ if $c_x(\s)=1$. In words, the constraint at $x$ is
satisfied if all the spin in $\cC_x$ are $0$. Note that $x_*$ is
unconstrained.
\end{definition}
\subsection{Oriented KCMs and main result}
\label{sec:graphical}
We give a general definition of the process to include a large class
of directed KCMs, such as the North-East model and higher
dimensional analogues. For constraining neighborhoods
$\{\cC_x\}_{x\in \L_n}$ we define the
associated directed KCM by the following graphical construction.  To
each $x\in\L_n$ we associate a mean one Poisson process and,
independently, a family of independent Bernoulli$(p)$ random variables
$\{s_{x,k} : k \in \N\}$. The occurrences of the Poisson process
associated to $x$ will be denoted by $\{t_{x,k} : k \in \N\}$.  We
assume independence as $x$ varies in $\L_n$. The probability
measure will be denoted by $\bbP^{\L_n}$. Notice that $\bbP^{\L_n}$-almost surely all
the occurrences $\{t_{x,k}\}_{k\in\N, x\in\L_n}$ are different.

Given $\h \in \O_n$ we construct a continuous time Markov chain
$\{\h(s)\}_{s\geq 0}$ on the probability space above, starting from
$\h$ at $t=0$, according to the following rules. At each time
$t_{x,n}$ the site $x$ queries the state of its own constraint $c_x$
(see \eqref{eq:constraint}).
If and only if the constraint is satisfied ($c_x = 1$) then $t_{x,n}$
is called a \emph{legal ring} and the configuration resets its
value at site $x$ to the value of the corresponding Bernoulli variable
$s_{x,n}$. 

The above construction gives rise to an irreducible, continuous time Markov chain,
reversible w.r.t. $\pi$, 
with generator
\begin{align}
  \label{eq:gen}
  \cL_n f(\s) = \sum_{x\in\L_n} c_x(\s) [\p_x(f) - f(\s)]\,,
\end{align}
where $\p_x$ denotes the conditional mean $\p(f\mid \{\s_y\}_{y\neq
  x})$. Irreducibility follows because we can invade $\L_n$ with $0$'s
starting from the unconstrained corner $x_*$.

Using a standard percolation argument \cite{Liggett2,Durrett} together with the fact
that the constraints $\{c_x\}_{x\in \L_n}$ are uniformly bounded and of finite range, it
is not difficult to see that the graphical construction can be
extended  without problems also to the infinite volume case.

For an initial distribution $\nu$ at $t=0$ the law and expectation of the process will be denoted by
$\bbP_\nu$ and $\E_\nu$ respectively.
In the sequel, we will write $\nu_t$ for the distribution of the chain
at time $t$,
\[
\nu_t[\eta] := \bbP_\nu(\h(t) = \eta).
\] 
If
$\nu$ is concentrated on a single configuration $\eta$ we will write $\bbP_\eta(\cdot)$. 
Note that, for each $i=1,\dots, n$,  the same graphical construction can be used to define the
process on $\O^{(i)}$ whose law is denoted by $\bbP_{\nu}^{(i)}(\cdot)$.

It follows from the graphical construction, and the fact that the
constraints are oriented, that given an $i\in\{1\ldots,n\}$ the evolution in $\cU_i$
is not influenced by the evolution above $\cU_i$. In particular, for any $\h \in \O_n$
and any event $\cA$ in the $\s$-algebra generated by
$\{\h_x(s)\}_{s\leq t,x\in\cU_i}$, 
\begin{align}
  \label{eq:cons-1}
  \bbP_\h(\cA) = \bbP_{\h^{(i)}}^{(i)}(\cA)\,.
\end{align}
In fact the same holds for any subset $U\subset \L_n$ with monotone surface, \ie
$x-\be_i\in U$ whenever $x\in U$ and $x-\be_i\in \L_n$, for each $i\in\{1,\ldots,d\}$.

We finish this section with definitions of the spectral gap and
mixing time of the process.
\begin{definition}[spectral gap]
\label{gap}
The \emph{spectral gap}, $\l^{(n)}$, of the infinitesimal generator \eqref{eq:gen} is the smallest
positive eigenvalue of $-\cL_n$ , and is given by the variational
principle
\begin{align}
  \label{eq:5}
  \l^{(n)} : = \inf_{\substack{f:\O_n \to \bbR \\ f \neq \textrm{const}}}\frac{\cD_n(f)}{\Var_n(f)},
\end{align}
where $\cD_n(f)= -\pi\left(f\cL_n f\right)$ is the Dirichlet form of
the process.\end{definition}
\begin{definition}[Mixing time]
The mixing time is defined in the usual way as
\begin{align}
  \label{eq:Tmix}
  T_{\rm mix}^{(n)} = \inf\{t>0 : \sup_{\nu\in\cP(\O_n)}\left\Vert\nu_t - \pi
  \right\Vert_{\rm TV} \leq 1/4\}\,,
\end{align}
where $\Vert\cdot\Vert_{\rm TV}$ denotes the total variation distance.
\end{definition}
\begin{remark}
It follows from \cite[Theorem 4.1]{CMRT} 
that there exists $0<p_0<1$ such that $\inf_n \l^{(n)} > 0$ 
whenever $p < p_0$. For the North-East model $p_0$ coincides with the
oriented percolation threshold. Moreover, for any $p\in (0,1)$, one has
$1/\l^{(n)}\le T_{\rm
  mix}^{(n)}\le c\, n^d/\l^{(n)}$ (see e.g. 
\cite{Levin2008}).
\end{remark}
Our main result reads as follows.

\begin{theorem}
\label{th:main}
 Assume $\inf_n \lambda^{(n)} > 0$, then 
 \begin{align}
   C\, n \leq T_{\rm mix}^{(n)} \leq C'\, n \log n
 \end{align}
 for some $C,C'> 0$ independent of $n$.
\end{theorem}
\begin{remark}
In the case of maximal constraints, \ie $\cC_x=\cK_x^*\cap \L_n$, 
inserting the indicator of the
configuration identically equal to $1$ as a test function in the
logarithmic Sobolev inequality \cite{Saloff} shows that
the logarithmic Sobolev constant $\a^{(n)}$ of the generator \eqref{eq:gen} 
 is
$O(1/n^d)$. Therefore, if \[
T^{(n)}_2:=\inf\{t\ge 0: \ \sup_\nu
\var\left(\frac {\nu_t}{\pi}\right)\le 1/4\},
\]
then for a universal constants $c>0$ and a constant $c'>0$ depending
only on $p$, 
\[
T^{(n)}_2\ge \frac{c}{\a^{(n)}}\ge c' n^d,
\]
where the first inequality follows from \cite [Corollary
2.2.7]{Saloff}. This observation shows that oriented models in
$\bbZ^d$ can be quite different from oriented models on rooted $k$-regular
trees. In the latter case it was recently proved
\cite{Martinelli:2012tp} that both the mixing time $T_{\rm mix}^{(n)}$
\emph{and} $T^{(n)}_2$
grow linearly in the depth $n$ of the tree whenever the relaxation
time $1/\l^{(n)}$ is $O(1)$.   
\end{remark}

\section{Proof of theorem \ref{th:main}}
\begin{proof}[Proof of the lower bound]
The lower bound on the mixing time is quite straightforward using finite speed of
propagation. Define  $\t^*$ as the first time the spin at 
site $x^*:=(n,\ldots,n)$ is $0$, and denote the configuration of all
$1$'s by $\1$. Using results in \cite{O,PS} there exists $c>0$ such
that 
\begin{align}
  \label{eq:3}
\Tmix{n} \geq c \bbE_{\1}\bigl[\t^*\bigr]\geq  \frac{c\,n}{2}
\bbP_\1\bigl(\t^*\ge n/2\bigr).
\end{align}
Clearly the event $\t^*< n/2$ requires the existence of a
path $\g=\{x_*=x^{(0)}, x^{(1)}, \dots, x^{(\ell)}=x^*\}$ and times
$t_0<t_1<\dots <t_\ell<n/2$ such that, for any $i\le \ell$, $x^{(i)}\in
\cC_{x^{(i+1)}}$ and $t_i$ is a
  legal ring for $x^{(i)}$. Standard Poisson large deviations show
  that the probability of the above event is $o(1)$ as $n\to \infty$.  
\end{proof}

\begin{proof}[Proof of the upper bound]
The proof of the upper bound is based on an iterative scheme. For
$i\leq n$, we find some time, $\D_i=O(\log(i))$, such that if the initial measure
$\nu$ has marginal $\nu^{(i-1)}$ equal to $\p^{(i-1)}$ on $\O^{(i-1)}$ (\ie below the
hyperplane $\cH_i$), then after time $\D_i$ the
marginal $\nu^{(i)}$ is very close to the equilibrium marginal $\pi^{(i)}$. This is the content of Lemma
\ref{lem:SingleDiag}. Then, starting from an arbitrary
initial measure, we can iterate the above result using the triangle
inequality for the variation distance and propagate the error.

\begin{lemma}[Mixing time for a single diagonal.]
\label{lem:SingleDiag}
There exists $c=c(q,d) > 0$ such that, for any initial measure $\nu$ with marginal on $\O^{(i-1)}$ equal to $\pi^{(i-1)}$, 
 \begin{align*}
   \left\Vert \nu^{(i)}_t - \pi^{(i)} \right\Vert_{\rm TV} \leq \epsilon \quad \textrm{for all} \quad t \geq c\,\log
(i/\epsilon)\ .
 \end{align*}
\end{lemma}
Before proving Lemma \ref{lem:SingleDiag} let us recall a useful
characterisation of the total variation distance (see for example \cite{Levin2008}), 
\begin{align*}
  \left\Vert\nu - \pi  \right\Vert_{\rm TV} =
  \frac{1}{2}\sup_{\|f\|_\infty \leq 1}\left| \nu(f) - \pi(f) \right|\,,
\end{align*}
where $\|f\|_\infty$ denotes the sup-norm.

\begin{proof}[Proof of Lemma \ref{lem:SingleDiag}]
Fix $i\leq n$ and a function $f$ depending only on the spin configuration in
$\cU_i$ and such that $\|f\|_\infty\le 1$. Without loss of generality
assume $\pi(f) = 0$. Given an initial measure $\nu$ with marginal
$\pi^{(i-1)}$ on $\O^{(i-1)}$
it follows from \eqref{eq:cons-1} that $\nu_t^{(i-1)}=\pi^{(i-1)}$ for all
$t\geq 0$. Moreover, conditioned on the history $\{
\sigma^{(i-1)}(s)\}_{s\leq t}$ in $\cU_{i-1}$, the spins
$\{\s_x(t)\}_{x\in \cH_i}$ on the
hyperplane $\cH_i$ evolve independently from each other. Each one goes
to equilibrium with rate one during the intervals of time in which its 
constraint $c_x$ is satisfied and stays fixed otherwise. 
In particular, conditioned on having had a legal ring at each $x\in \cH_i$ before time $t$, the
distribution of $\s_{\cH_i}(t)$ is $\pi$. 
These simple observations gives rise to an upper-bound on
the expectation of $f$ at time $t$ as follows.

Let $\t^{(i)}$ be the first time there has been at least one legal ring
on each $x\in \cH_i$. Following the argument above we have,
\begin{align*}
\left\vert\Ep{\nu}{f\left(\sigma^{(i)}(t)\right)}\right\vert =&
\left\vert\ \bbE_\nu\left(\Ep{\nu}{f\left(\sigma^{(i)}(t)\right)\mid \big\{
\sigma^{(i-1)}(s)\big\}_{s\leq t}}\right) \right\vert\\
\leq& \left\vert\ \bbE_\nu\left(\Ep{\nu}{f\left(\sigma^{(i)}(t)\right)  \1_{\{\t^i<t\}}\mid \big\{
\sigma^{(i-1)}(s)\big\}_{s\leq t}} \right)\right\vert +\\
&+\left\vert\ \bbE_\nu\left(\Ep{\nu}{f\left(\sigma^{(i)}(t)\right)
      \1_{\{\t^i\ge t\}}\mid \big\{
\sigma^{(i-1)}(s)\big\}_{s\leq t}}\right)\right\vert\\
\leq&  \left\vert\ \bbE_\nu\left(
    \pi_{\cH_{i}}\left(f\right) \bbP_\nu\left(\t^i<t \mid
      \big\{\sigma^{(i-1)}(s)\big\}_{s\leq t}\right)\right)\right\vert +\\
&+  \left\vert\ \bbE_\nu\left(\Ep{\nu}{f\left(\sigma^{(i)}(t)\right)
      \1_{\{\t^i\ge t\}}\mid \big\{
\sigma^{(i-1)}(s)\big\}_{s\leq t}}\right) \right\vert\\
\leq & 2\Ep{\pi}{ \bbP_\nu\left(\t^i\ge t\mid \big\{
\sigma^{(i-1)}(s)\big\}_{s\leq t}\right)}.
\end{align*}
In the second inequality
above we used the strong Markov property
together with the observation that the distribution of $\s_{\cH_i}(t)$
conditioned on $\{\t^{(i)}<t\}$ and $\big\{
\sigma^{(i-1)}(s)\big\}_{s\leq t}$ is $\pi$.  In the third
inequality  we used the fact that $\bbE_\nu\left(\pi_{\cH_{i}}(f)\right)=\pi(f)=0$.

To bound the final term above we denote the number of rings of the
Poisson clock at site $x$ during the
 set $B\subset [0,t]$ by
$N_x(B) =\lvert \{t_{x,k}\}_{k\geq 0}\cap B \rvert$, and we define the
set of \emph{legal times} at site $x$ before $t$ as $\cG(x,t) := \{ s < t :
c_x(\sigma(s)) =1 \}$. By construction, for $x\in\cH_i$ the set $\cG(x,t)$  depends only on $\left\{
\sigma_y(s)\right\}_{s\leq t},\ y\in \cC_x$. 
Using the notation above,
\[
\bbP_\nu\left(\t^{(i)}\ge t\mid \big\{\sigma^{(i-1)}(s)\big\}_{s\leq
    t}\right) \leq 
\sum_{x\in\cH_i}\bbP\left( N_x(\cG(x,t)) = 0 \mid \big\{
\sigma^{(i-1)}(s)\big\}_{s\leq t}\right).
\]
By construction $N_x(\cG(x,t))$ is a Poisson random variable of mean
$ |\cG(x,t)|$. Thus
\[
\sum_{x\in\cH_i}\bbP\left( N_x(\cG(x,t)) = 0 \mid \big\{
\sigma^{(i-1)}(s)\big\}_{s\leq t}\right)= \sum_{x\in\cH_i}e^{-|\cG(x,t)|}.
\]
In conclusion, 
\begin{align*}
\left\vert \Ep{\nu}{f\left(\s^{(i)}(t)\right)} \right\vert &\leq
2\sum_{x\in\cH_i}\Ep{\pi}{e^{-|\cG(x,t)|} }\leq 2 i^{d-1}
\max_{y\in \L_n}\Ep{\pi}{e^{-|\cG(y,t)|}}.
\end{align*}
We estimate $\Ep{\pi}{e^{-|\cG(y,t)|}} $ using a Feynman-Kac approach (a similar
method has been used to bound the persistence function in \cite{CMRT}).
The total time that site $y$ satisfies its constraints is 
\begin{align*}
  |\cG(y,t)| = \int_0^t c_y\left(\s(s)\right)ds.
\end{align*}
On $L^2(\pi)$ we define the self adjoint operator $H := \cL - V$ with $V(\sigma):=c_y(\sigma )$.
The Feynman-Kac formula allows us to rewrite the expectation
$\Ep{\pi}{e^{-\int_{0}^{t} V}}$  as
$\left\langle\mathbf{1},e^{tH}\mathbf{1} \right\rangle_\p$ (where
$\left\langle\cdot,\cdot\right\rangle_\p$ denotes the
inner-product in $L^2(\pi)$).
Thus, if $\beta_y$ is the supremum of the spectrum of $H$ we have
\begin{align}
\label{eq:spec}
\Ep{\pi}{e^{-|\cG(y,t)|} }\leq e^{t\beta_y}\ .
\end{align}
In order to complete the proof of the Lemma it remains to show that $\beta_y < 0$.

We decompose each $L^2(\pi)$-norm one function $\phi$ in the domain of $H$ as $ \phi=
\alpha \mathbf{1} + g$, for some mean zero function $g$.
So $\left\langle \mathbf{1},g\right\rangle_\p = 0$ and $\alpha^2 +
\langle g,g\rangle_\p = 1$. Then,
\begin{align}
\label{eq:H}
  \left\langle \phi,H \phi \right\rangle = \left\langle g, \cL g \right\rangle_\p -
\alpha^2\left\langle \mathbf{1},V\fc\right\rangle_\p  - 2\alpha \left\langle \mathbf{1},V g
\right\rangle_\p - \left\langle g,V g \right\rangle_\p
\end{align}
We proceed by bounding $\left\langle \phi,H \phi\right\rangle_\p$.
Using \eqref{eq:H} together with the Cauchy-Schwarz inequality we get
\begin{align*}
\left\langle \phi,H \phi \right\rangle &\leq -\l^{(n)}\, \langle g, g\rangle_\p -
q^d\Big( \alpha^2 + 2 \bbE_\p\left(g\mid V = 1\right)\alpha  + \bbE_\p\left(g^2 \mid V = 1\right) \Big) \\
&\leq  -\delta\,\l^{(n)} -q^d \Big(|\a|-\bbE_\p\left(g^2 \mid V =
  1\right)^{1/2}\Big)^2\\
 &\leq  -\delta\,\l^{(n)}
\end{align*}
where $\d=1-\a^2$. 
Again from \eqref{eq:H}, dropping the
last term on the right hand side and applying the Cauchy-Schwarz
inequality we get
\begin{align*}
\left\langle \phi,H \phi \right\rangle &\leq -\l^{(n)}\ \langle g, g\rangle_\p - \alpha^2 q^d  + 2
\vert \alpha \vert \left(\langle g , g \rangle_\p
q^d(1-q^d)\right)^{1/2}\\
&\leq  (\d -1 )q^d + 2\left(\delta
q^d(1-q^d)\right)^{1/2} .
\end{align*} 
Thus 
\[
\beta_y \le -c_0(q):=\min\left(-\delta\,\l^{(n)}, (\d -1 )q^d + 2\left(\delta
q^d(1-q^d)\right)^{1/2} \right).
\]
Finally we observe that $c_0(q)>0$ because 
\[
(\d -1 )q^d + 2\left(\delta
q^d(1-q^d)\right)^{1/2} \le  -\frac{q^d}{2}\quad\textrm{if}\quad \d
\leq \frac{2(1-\sqrt{1-q^d})-q^d}{4q^d}.
\]
In conclusion 
\begin{align*}
\left\vert \Ep{\nu}{f\left(\s^{(i)}(t)\right)} \right\vert
\leq 2 i^d e^{-c_0(q) t}\le \epsilon,
\end{align*}
for $t\geq c \log(i/\epsilon)$ for some $c=c(d,q)$.
\end{proof}



We now prove Theorem \ref{th:main} by iterating the previous Lemma and
propagating the error.
We use the following iterative scheme
\begin{align}
  t_{i} = t_{i-1} + \D_i, \quad t_0=0\,,
\end{align}
where $\D_i(\epsilon) = \frac{1}{c} \log(i^3/\epsilon)$ with $c$ as in the
Lemma.
We now show by induction that starting from an arbitrary initial distribution $\nu$ 
\begin{align}
  \label{eq:H0}
 2 \left\Vert \nu^{(i)}_t - \p^{(i)} \right\Vert_{\rm TV} \leq \hat\epsilon_{i}
:=  \epsilon\sum_{j=1}^{i} \frac{1}{j^2}\quad \textrm{for all} \quad  t
\geq  t_{i}\ .
\end{align}
The case $i=1$ \eqref{eq:H0} is an immediate consequence of the fact that
the corner $x_*=(1,\dots,1)$ goes to equilibrium with rate one.
Assume \eqref{eq:H0} holds for all $i \leq k-1$. It is clear that
$\min_\s\nu_t(\s)>0$ for any $t>0$, so that we may define an auxiliary
probability measure $\mu$ with marginal $\mu^{(k-1)}=\pi^{(k-1)}$ by
\[
\mu[\s] := \pi^{(k-1)}\left[\s^{(k-1)}\right]\, \nu_{t_{k-1}}\left[\s \tc \s^{(k-1)}
\right].
\]
Fix a function $f$ depending  only on the configuration in
$\cU_{k}$ with $\|f\|_{\infty}\le 1$. Again, without loss of generality,
assume $\pi(f) = 0$. Then 
\begin{align}
  \label{eq:1}
  \Big\vert \Ep{\nu}{f\left(\s(t_{k})\right)} \Big\vert &=
  \Big\vert \Ep{\nu_{t_{k-1}}}{f\left(\s(\D_{k})\right)} \Big\vert\nonumber\\
  &\leq \Big\vert \Ep{\nu_{t_{k-1}}}{f\left(\s(\D_{k})\right)} -
    \Ep{\mu}{f\left(\s(\D_{k})\right)}\Big\vert + \frac{\e}{k^2}\,.
\end{align}
where we used the Markov property in the first line and Lemma
\ref{lem:SingleDiag} together with the triangle inequality in the
second line.
The two measures $\mu$ and $\nu_{t_{k-1}}$ have the same marginal on
$\cH_k$, so we may reduce the remaining term to something that can be
dealt with using the inductive hypothesis as follows. For any $\eta\in \O^{(k-1)}$
let
\begin{align*}
  F\left(\eta\right) := \Ep{\nu_{t_{k-1}}}{f\left(\s(\D_{k})\right)\tc \s^{(k-1)}=\eta
  }.
\end{align*}
Clearly $\Vert F \Vert_{\infty} \leq 1$, so by the definition of
$\mu$ and the
inductive hypothesis \eqref{eq:H0},
\begin{align*}
  \Big\vert \Ep{\nu_{t_{k-1}}}{f\left(\s(\D_{k})\right)} -
    \Ep{\mu}{f\left(\s(\D_{k})\right)}\Big\vert = \left| \nu_{t_{k-1}}
    \left( F \right) - \pi\left( F \right) \right|\le \hat\e_{k-1}\,.
\end{align*}
In conclusion the r.h.s. of \eqref{eq:1} is bounded from above by 
\begin{align*}
  \Big\vert \Ep{\nu}{f\left(\s(t_{k})\right)} \Big\vert &=
  \hat\e_{k-1}+ \frac{\e}{k^2} = \hat\e_k\,.
\end{align*}
Since $\cU_{2n-1} \supseteq \L_n$, it follows that
\begin{align*}
2  \left\Vert \nu_t - \p \right\Vert_{\rm TV} \leq \hat\epsilon_{2n-1}
=  \epsilon\sum_{j=1}^{2n-1} \frac{1}{j^2} \leq 2 \e \quad \textrm{for
  all} \quad t \geq t_{2n-1}
\end{align*}
and
\begin{align*}
t_{2n-1} = \sum_{j=1}^{2n-1}\D_j \leq \frac{1}{c} \int_{0}^{2n} \log(
x^3/\epsilon) d x \le c'n\log(n)
\end{align*}
for some $c'>0$.
\end{proof}

\section{Acknowledgments} We would like to thank O. Blondel for a very
careful reading of the paper.

\bibliography{NE}
\bibliographystyle{plain}


\end{document}